\documentclass[english]{amsart}

\usepackage{amsmath, amsthm}
\usepackage{amssymb}
\usepackage{amscd}
\usepackage{latexsym}
\usepackage[latin1]{inputenc}
\usepackage[english]{babel}
\usepackage[shortlabels]{enumitem}
\usepackage{xr}

\usepackage{xypic}
\usepackage{graphicx}
\usepackage{changebar,color}

\newtheorem{theorem}{Theorem}[section]
\newtheorem{corollary}[theorem]{Corollary}
\newtheorem{lemma}[theorem]{Lemma}

\newtheorem{proposition}[theorem]{Proposition}

\theoremstyle{remark}

\theoremstyle{definition}
\newtheorem{example}[theorem]{Example}
\newtheorem{definition}[theorem]{Definition}

\numberwithin{equation}{subsection}

\def\Z{\mathbb{Z}}

\def\Q{\mathbb{Q}}
\def\R{\mathbb{R}}
\def\C{\mathbb{C}}

\newcommand{\Gm}{{{\mathbb G}_m}}

\newcommand{\der}{{\rm der}}

\newcommand{\Sd}{ \mathbb S}

\newcommand{\ra}{\rightarrow}

\newcommand{\lra}{\longrightarrow}

\newcommand{\hlra}{{\lhook\joinrel\longrightarrow}}

\title{Two results on the Hodge structure of complex tori}

\author{Fran\c{c}ois Charles}

\dedicatory{Pour Olivier, avec amitié, un petit appendice à son beau livre ``Tores et Variétés Abéliennes complexes''}

\begin{document}

\maketitle

\begin{abstract}
We prove two results regarding Hodge structures appearing in the cohomology of complex tori. First, we prove that if a polarizable Hodge structure appears in the cohomology of a complex torus $T$, it appears in the cohomology of an abelian variety isomorphic to a subquotient of $T$. Second, we prove a universality result for the Kuga-Satake construction applied to Hodge structures of $K3$ type that might not be polarized.
\end{abstract}

\section{Introduction}

The theme of this note is the Hodge structures that appear in the cohomology of complex tori. It is well-known that the functor which associates to a complex torus $T$ its first Betti cohomology group $H^1(T, \Q)$ together with its natural Hodge structure of weight $1$ is an equivalence of category between the category of complex tori up to isogeny and the category of rational Hodge structures of type $\{(1, 0), (0,1)\}$. Furthermore, such a rational Hodge structure comes from an abelian variety if and only if it admits a polarization, namely, a rational bilinear form that satisfies the Hodge-Riemann positivity relations. 

In general, we say that a Hodge structure $V$ is abelian if it appears as a direct factor, up to a Tate twist, of the cohomology of an abelian variety. Concretely, this means that $V$ appears as a direct factor of a Hodge structure of the form 
$$H^1(A, \Q)^{\otimes a}\otimes (H^1(A, \Q)^{\vee})^{\otimes b}\otimes\Q(c),$$
where $a, b$ and $c$ are positive integers with $a, b\geq 0$. Abelian Hodge structures are described in terms of their Mumford-Tate groups in \cite[Section 1]{Milne94}.

Our first result shows that for a polarizable Hodge structure $V$ to be of abelian type, it suffices for $V$ to appear in the cohomology of a complex torus $T$, whithout assuming that $T$ is an abelian variety. More precisely, we prove the following (see Theorem \ref{theorem:torus-abelian}):

\begin{theorem}
Let $V$ be a pure polarizable Hodge structure. Let $W$ be a Hodge structure of type $\{(0, 1),\,(1, 0)\}$ such that $V$ is isomorphic to a subquotient of the Hodge structure
$$W^{\otimes a}\otimes (W^\vee)^{\otimes b}\otimes\Q(c)$$
for some integers $a, b, c$ with $a, b\geq 0.$
Then there exists a polarizable Hodge structure $W'$ which is isomorphic to a direct sum of subquotients of $W$, and an injection of Hodge structures
$$V\hlra W'^{\otimes a'}\otimes (W'^\vee)^{\otimes b'}\otimes\Q(c')$$
for some integers $a', b', c'$ with $a', b'\geq 0.$
\end{theorem}

In particular, those polarizable Hodge structures that appear in the cohomology of complex tori always come from algebraic varieties.

\bigskip

An instance of Hodge structures that appear in the cohomology of complex tori is given by the Kuga-Satake construction, first introduced in \cite{Deligne72}, that associates a Hodge structure $W$ of type $\{(0, 1),\,(1, 0)\}$ to any Hodge structure $V$ of $K3$ type endowed with a suitable quadratic form. The Hodge structure $W$ is polarizable if the quadratic form is a polarization, and $V$ is always a sub-Hodge structure of $\mathrm{End}(W)$. We prove in Theorem \ref{theorem:KS-univ} that $W$ satisfies a universality property: if $V$ is general enough and $W'$ is a Hodge structure of type $\{(0, 1),\,(1, 0)\}$  such that there exists 
an embedding
$$V\hlra W'^{\otimes a}\otimes (W'^\vee)^{\otimes b}\otimes\Q(c)$$
for some $a, b, c$ with $a, b$ nonnegative, then $W'$ shares a subquotient with the Kuga-Satake variety.

In the situation where $V$ is polarized, this universality theorem is folklore and stated for instance in \cite[Proposition 6]{vanGeemenVoisin16}, and essentially contained in \cite[1.3]{Deligne79}, following results of Satake \cite{Satake65}, see also \cite{CharlesKSuniv}. The result is also alluded to by Deligne in his paper \cite{Deligne72}. 

\bigskip

Our arguments for the proof of Theorem \ref{theorem:torus-abelian} and Theorem \ref{theorem:KS-univ} rely on the use of Mumford-Tate groups and basic elements of the theory of reductive groups. In section 2, we gather notation and elementary or folklore results. Section 3 is devoted to the proof of Theorem \ref{theorem:torus-abelian}, and section 4 proves Theorem \ref{theorem:KS-univ}.

\noindent\textbf{Acknowledgements.} We are grateful to Claire Voisin for raising the question of the universality of the Kuga-Satake construction without polarizations and related discussions. This project has received funding from the European Research Council (ERC) under the European Union's Horizon 2020 research and innovation programme (grant agreement No 715747).

\section{Notation and preliminary results}

\subsection{Hodge structures}

\subsubsection{}
We denote by $\Sd$ the \emph{Deligne torus}, namely, the real algebraic group defined as the Weil restriction of the multiplicative group $\Gm_{,\R}$ to $\R$. By definition, we have:
$$\Sd(\R)=\C^*\,\,\,\mathrm{and}\,\,\,\Sd(\C)=\C^*\times\C^*.$$

We denote by 
$$w : \Gm_{,\R}\lra \Sd$$
the morphism corresponding at the level of real points to the inclusion of $\R^*$ in $\C^*$, and by 
$$\mu : \Gm_{,\C}\lra \Sd_\C\simeq \Gm_{,\C}\times \Gm_{,\C}$$
the cocharacter corresponding to the morphism $z\mapsto (z, 1).$

\subsubsection{}
Let $V$ be a finite-dimensional vector space over $\Q$. A \emph{Hodge structure} on $V$ is the datum of a 
bigrading of the complex vector space $V_\C$:
$$V_\C=\bigoplus_{p, q\in\Z} V^{p,q}$$
such that for all integers $p, q$, the spaces $V^{p, q}$ and $V^{q, p}$ are conjugate and the spaces 
$$\bigoplus_{p+q=k} V^{p, q}$$
are defined over $\Q$ for all integers $k$. Given a subset $S$ of $\Z^2$, we say that $V$ is of type $S$ if $V^{p, q}=0$ when $(p, q)$ does not lie in $S$.

The Hodge structure is pure of some weight $k$ if $V^{p, q}=0$ when $p+q\neq k$. If $n$ is an integer, the Tate Hodge structure $\Q(n)$ is the unique pure Hodge structure of type $(-n, -n)$ on the rational vector space $\Q$ \footnote{Since we will need to use Tate Hodge structures of fractional weight, we find it more convenient to drop the factor $(2\pi i)^n$ in the definition of Tate Hodge structures. Of course, this does not change the isomorphism class of these Hodge structures.}.

The datum of an action of the torus $\Sd$ on the real vector space $V_\R$ is equivalent to a bigrading of the complex vector space $V_\C$:
$$V_\C=\bigoplus_{p, q\in\Z} V^{p,q}$$
such that for all integers $p, q$, the spaces $V^{p, q}$ and $V^{q, p}$ are conjugate -- here $V^{p, q}$ is the subspace of $V_\C$ on which $\C^*$ acts through $z^{-p}\overline{z}^{-q}.$ Via the morphism $w : \Gm_{,\R}\ra\Sd$, an action of $\Sd$ on $V_\R$ induces an action of $\Gm_{, \R}$ on $V_\R$. This latter action is defined over $\Q$ if and only if the spaces 
$$\bigoplus_{p+q=k} V^{p, q}$$
are defined over $\Q$ for all integers $k$, i.e., if and only if the bigrading above defines a Hodge structure on $V$. Given an action $h : \Sd\ra GL(V_\R)$, the corresponding Hodge structure on $V$ is pure if and only if the morphism
$$h\circ w : \Gm_{,\R}\lra GL(V_\R)$$
factors through the center of $GL(V)_\R$.

If $V$ is a pure Hodge structure of weight $k$, a \emph{polarization} of $V$ is a morphism of Hodge structures 
$$\phi : V\otimes_\Q V\lra\Q(-k)$$
such that the bilinear form
$$V_\R\otimes V_\R\lra \R,\,\, (x, y)\mapsto \phi_\R(x, h(i) y)$$
is definite positive.

\subsubsection{} We denote by $MT(V)$ the smallest algebraic subgroup $G$ of the algebraic group $GL(V)$ over $\Q$ such that $G_\R$ contains the image of $h$. The group $MT(V)$ is the \emph{Mumford-Tate group} of the Hodge structure $V$. Let
$$h : \Sd\lra MT(V)_\R$$
be the morphism that defines the Hodge structure on $V$. We denote by 
$$w_h : \Gm_{,\Q}\lra MT(V)$$
the morphism induced by $h\circ W$. It is a central cocharacter if and only if $V$ is pure.

\subsubsection{} We will need to consider \emph{fractional Hodge structures}. A fractional Hodge structure on a finite-dimensional vector space $V$ over $\Q$ is defined as the datum of a bigrading of the complex vector space $V_\C$:
$$V_\C=\bigoplus_{p, q\in\Q} V^{p,q}$$
such that for all rational numbers $p, q$, the spaces $V^{p, q}$ and $V^{q, p}$ are conjugate and the spaces 
$$\bigoplus_{p+q=k} V^{p, q}$$
are defined over $\Q$ for all rational numbers $k$. Given a subset $S$ of $\Q^2$, we say that $V$ is of type $S$ if $V^{p, q}=0$ when $(p, q)$ does not lie in $S$.

 If $a$ is a rational number, the fractional Tate Hodge structure $\Q(a)$ is the unique pure Hodge structure of type $(-a, -a)$ on the rational vector space $\Q$.

Fractional Hodge structures may be understood in terms of the Deligne torus as follows. Let $T$ be an algebraic torus. We denote by $\widetilde T$ the \emph{universal covering} of $T$, namely, the projective system $(T_n)_{n\in \mathbb N\setminus\{0\}}$ where $T_n=T$ for all positive $n$, $\mathbb N\setminus\{0\}$ is ordered b divisibility, and the transition map 
$$T_{mn}\lra T_{n}$$
is $x\mapsto x^m.$ A \emph{fractional morphism} from $T$ to an algebraic group $G$ is a morphism 
$$\phi : \widetilde T\lra G.$$
Such a morphism may be represented by an actual morphism 
$$\phi_n : T_n\lra G.$$

We denote by 
$$\widetilde w : \widetilde{\Gm}_{,\R}\lra \widetilde\Sd$$
the morphism induced by $w$.

As in the case of usual Hodge structures, fractional Hodge structures on $V$ correspond bijectively to those fractional morphisms
$$\widetilde h : \widetilde\Sd\lra GL(V)_\R$$
such that $\widetilde h\circ\widetilde w$ is defined over $\Q$. 

Given a fractional morphism $\widetilde h$ as above, we may find a positive integer $n$ and a commutative diagram
\[
\xymatrix{
\widetilde \Sd\ar[r]^{x\mapsto x^n}\ar[dr]^{\widetilde h_n} & \widetilde\Sd\ar[d]^{\widetilde h}\\
& GL(V)_\R
}
\]
in which $\widetilde h_n$ comes from an actual morphism $h_n : \Sd\ra GL(V)_\R$. The smallest algebraic subgroup $G$ of $GL(V)$ over $\Q$ such that $G_\R$ contains the image of $h_n$ does not depend on the choice of $n$, we will denote it again by $MT(V)$ and call it the Mumford-Tate group of $V$.


\subsubsection{} Fractional morphisms and fractional Hodge structures appear through the following elementary result.

\begin{proposition}\label{proposition:fractional-lift}
Let $k$ be a field of characteristic zero, let $T$ be an algebraic torus over $k$, and let 
$$h: T\lra H$$
be a morphism of algebraic groups over $k$. Let 
$$p : G\lra H$$
be an isogeny. Then $h$ lifts uniquely to a fractional morphism 
$$\widetilde h : \widetilde T\lra G.$$
\end{proposition}

\begin{proof}
The unicity statement is clear. To prove existence, we may replace both $T$ and $H$ with the image of $T$ in $H$, and $G$ with the identity component of the preimage of $h(T)$ in $G$.

Since $p$ has finite kernel, there exists a positive integer $N$ such that the kernel $K$ of $p$ is contained in the $N$-torsion subgroup of $G$. In particular, the morphism 
$$G\lra G,\,\,x\mapsto x^N$$
factors through a morphism $\phi :  T\ra G$: the composition
$$G\stackrel{p}{\lra} T\stackrel{\phi}{\lra} G$$
is $x\mapsto x^N.$ Consider the composition
$$\psi : T \stackrel{\phi}{\lra} G\stackrel{p}{\lra} T.$$
Then 
$$\psi\circ p : G\lra T$$
is the morphism $p\circ(\phi\circ p)$, i.e.
$$G\stackrel{x\mapsto x^N}{\lra} G\stackrel{p}{\lra} T.$$
Consider the morphism
$$\psi_N : T\lra T,\,\, x\mapsto x^N.$$
Then $\psi\circ p=\psi_N\circ p$, so that $\psi=\psi_N$ as $p$ is an epimorphism by e.g. \cite[Lemma 2.1]{Brion17}.

In particular, we obtain a commutative diagram
\[
\xymatrix{
T\ar[d]^{\psi_N}\ar[r]^\phi & G\ar[d]^p\\
T\ar[r]^{\mathrm{Id}_T} & T,
}
\]
and $\phi$ defines a fractional lift of $\mathrm{Id}_T$ to $G$.
\end{proof}

\subsection{The reductive quotient of a Mumford-Tate group and the semisimplification of a Hodge structure}

Let $V$ be a Hodge structure. Recall that $V$ is said to be simple if all sub-Hodge structures of $V$ are isomorphic to $0$ or $V$ itself. We say that $V$ is semisimple if $V$ is isomorphic to a direct sum of simple Hodge structures. It is well-known that polarizable Hodge structures are semisimple and that their Mumford-Tate group is reductive. 

In general, there exists an increasing filtration
$$0=V_0\subset\ldots\subset V_n=V$$
of $V$ by sub-Hodge structures such that, for all $i$ between $0$ and $n-1$, the quotient Hodge structure
$$V_{i+1}/V_i$$
is simple. The direct sum
$$V^{\mathrm{ss}}:=\bigoplus_{i=0}^{n-1} V_{i+1}/V_i$$
is a semisimple Hodge structure by construction, and it is readily checked that its isomorphism class does not depend on the choice of the filtration $V_\bullet.$

Let $G$ be the Mumford-Tate group of $V$, and let $h : \Sd\ra G_\R$ be the morphism defining the Hodge structure on $V$.

\begin{proposition}\label{proposition:reductive-semisimple}
The Hodge structure $V$ is semisimple if and only if the Mumford-Tate group $G$ is reductive.
\end{proposition}

\begin{proof}
The sub-Hodge structures of $V$ are exactly the $G$-invariant subspaces of $V$. As a consequence, $V$ is semisimple as a Hodge structure if and only if $V$ is semisimple as a representation of $G$. 

If $V$ is semisimple, then the unipotent radical of $G$ acts trivially on $V$. Since $G$ acts faithfully on $V$, this implies that $G$ is reductive. 

Conversely, if $G$ is reductive, then any representation of $G$ is semisimple, so that $V$ is semisimple.
\end{proof}

We may describe the Mumford-Tate group of the semisimplification of $V$ in general.

\begin{proposition}
The Mumford-Tate group of $V^{\mathrm{ss}}$ may be identified to the reductive quotient $G^{\mathrm{red}}$ of $G$. Under this identification, the morphism $\Sd\ra G^{\mathrm{red}}_\R$ defining the Hodge structure on $V^{\mathrm{ss}}$ is the composition of $h$ with the quotient map $G_\R\ra G_\R^{\mathrm{red}}.$
\end{proposition}

\begin{proof}
Consider an increasing filtration
$$0=V_0\subset\ldots\subset V_n=V$$
of $V$ by sub-Hodge structures such that, for all $i$ between $0$ and $n-1$, the quotient Hodge structure
$$V_{i+1}/V_i$$
is simple. Then the spaces $V_i$ are invariant under the action of $G$ on $V$.

Let $U$ be the unipotent radical of $G$. By assumption, the representation of $G$ on $V_{i+1}/V_i$ is simple for all $i$, so that $U$ acts trivially on $V_{i+1}/V_i$. As a consequence, the action of $G$ on $V$ induces an action of the reductive quotient $G/U=G^{\mathrm{red}}$ on $V^{\mathrm{ss}}$. By functoriality, the Hodge structure on $V^{\mathrm{ss}}$ is defined by this action and  the morphism 
$$h^{\mathrm{ss}} : \Sd\ra G^{\mathrm{red}}_\R$$ 
defined as the composition of $h$ with the quotient map $G_\R\ra G_\R^{\mathrm{red}}.$

Since $h$ has dense image in $G$, $h^{\mathrm{ss}}$ has dense image in $G^{\mathrm{red}}_\R.$ To prove that $G^{\mathrm{red}}$ is the Mumford-Tate group of $V$, it remains to prove that the representation of $G^{\mathrm{red}}$ on $V^{\mathrm{ss}}$ is faithful, i.e., that the kernel of the action of $G$ on $\bigoplus_{i=0}^{n-1} V_{i+1}/V_i$ is the unipotent radical $U$. Clearly, this kernel is normal and unipotent, so that it is contained in $U$, which finishes the proof.
\end{proof}

\section{Hodge structures coming from complex tori are of abelian type}

\subsection{Statement of the theorem}

The main theorem of this section is the following.

\begin{theorem}\label{theorem:torus-abelian}
Let $V$ be a pure polarizable Hodge structure. Let $W$ be a Hodge structure of type $\{(0, 1),\,(1, 0)\}$ such that $V$ is isomorphic to a subquotient of the Hodge structure
$$W^{\otimes a}\otimes (W^\vee)^{\otimes b}\otimes\Q(c)$$
for some integers $a, b, c$ with $a, b\geq 0.$
Then there exists a polarizable direct factor $W'$ of the Hodge structure $W^{\mathrm{ss}}$ and an injection of Hodge structures
$$V\hlra W'^{\otimes a'}\otimes (W'^\vee)^{\otimes b'}\otimes\Q(c')$$
for some integers $a', b', c'$ with $a', b'\geq 0.$

%
%

\end{theorem}

In particular, $V$ appears in the cohomology of an abelian variety, so we obtain the following corollary.

\begin{corollary}
Let $V$ be a polarizable Hodge structure. Let $W$ be a Hodge structure of type $\{(0, 1),\,(1, 0)\}$ such that $V$ is isomorphic to a subquotient of the Hodge structure
$$W^{\otimes a}\otimes (W^\vee)^{\otimes b}\otimes\Q(c)$$
for some integers $a, b, c$ with $a, b\geq 0.$ Then $V$ is isomorphic to a direct factor of a Hodge structure arising from the cohomology of an algebraic variety.
\end{corollary}

%
%
%
%
%
%
%

\subsection{Mumford-Tate groups of Hodge structures of weight $1$}\label{subsection:MT-ab}

\subsubsection{} We start with an elementary result on Hodge structures of weight $1$.

\begin{proposition}\label{proposition:tensor-abelian}
Let $W$ be a pure Hodge structure of type $\{(0, 1),\,(1, 0)\}$. Let $W_1, \ldots, W_k$ be fractional Hodge structures and assume that there is an isomorphism of fractional Hodge structures 
$$W\simeq W_1\otimes\ldots\otimes W_k.$$
Then there exists a unique integer $i$ between $1$ and $k$ with the following property: for all $j\neq i$, there exists a rational number $a_j$ such that $W_j$ is a direct sum of copies of the fractional Hodge structure $\Q(a_j)$. Furthermore, there exists a rational number $a_i$ such that $W_i$ is a fractional Hodge structure of type $\{(a_i+1, a_i),\, (a_i, a_i+1)\}.$
\end{proposition}

\begin{proof}
For any $i$ between $1$ and $k$, let $S_i\subset\Q^2$ denote the set of pairs of rational numbers $(p, q)$ such that $W_i^{p, q}\neq 0$. For any $(p, q)\in \Q^2$, $W^{p, q}$ is nonzero if and only if we may write 
$$(p, q)=(p_1+\ldots + p_k, q_1+\ldots+q_k),$$
where the pairs $(p_i, q_i)$ run through the elements of $S_i$. In particular, the sum $p_1+\ldots + p_k$ takes exactly the values $0$ and $1$. 

This implies that there exists a unique integer $i$ between $1$ and $k$ with the following property: for all $j\neq i$, there exists a unique rational number $a_j$ such that, for any $(p, q)\in S_j$, $p=a_j$. This implies that $W_j$ has type $\{(a_j, a_j)\}$, so that $W_j$ is a direct sum of copies of the fractional Hodge structure $\Q(a_j)$.

It follows immediately that $W_i$ is a fractional Hodge structure of type $\{(a_i+1, a_i),\, (a_i, a_i+1)\}.$
\end{proof}

\subsubsection{} Let $W$ be a Hodge structure of type $\{(0, 1),\,(1, 0)\}$, and let $G$ be the Mumford-Tate group of $W$. We assume that $G$ is reductive. By Proposition \ref{proposition:reductive-semisimple}, this is equivalent to $W$ being semisimple.

Consider connected normal subgroups $T$, $G_1, \ldots, G_r$ of $G$ that commute pairwise such that the following two conditions are satisfied: 
\begin{enumerate}
\item the multiplication map induces a central isogeny 
$$p : T\times G_1\times\ldots\times G_r\lra G;$$
\item for any $i$ between $1$ and $r$, there are no nontrivial characters $G_i\ra\Gm$.
\end{enumerate}
Note that the second condition above is satisfied if $G_i$ is simple or $G_i(\R)$ is compact.

Let 
$$h : \Sd\lra G$$
be the morphism corresponding to the Hodge structure $W$, and let 
$$\widetilde h : \widetilde \Sd\lra T\times G_1\times\ldots\times G_r$$
be the unique fractional lift of $h$ to $T\times G_1\times\ldots\times G_r$ provided by Proposition \ref{proposition:fractional-lift}. We write
$$\widetilde h=\widetilde h_0\,\widetilde h_1\ldots \widetilde h_r,$$
where $\widetilde h_0 : \widetilde \Sd\ra T$ and $\widetilde h_i : \widetilde \Sd\ra G_i$ for $1\leq i\leq r$ are the components of $\widetilde h$.

\begin{proposition}\label{proposition:unique-action}
Let $W'$ be a simple factor of the Hodge structure $W$. Then the following statements hold:
\begin{enumerate}[(i)]
\item there exists at most one integer $i$ in $\{1,\ldots, r\}$ such that $G_i$ acts nontrivially on $W'$. 
\end{enumerate}

Assume that $G_i$ acts nontrivially on $W'$.

\begin{enumerate}[(i)]
\setcounter{enumi}{1}
\item There exists a rational number $a$ such that the fractional Hodge structure on $W'$ induced by $\widetilde h_i$ and the action of $G_i$ on $W'$ is of type 
$$(\{(a+1, a), \, (a, a+1)\}.$$
\item If $G_i$ is simple, the natural surjection 
$$\pi : G=MT(W)\lra MT(W')$$
restricts to a central isogeny of $G_i$ onto its image, and it induces an isomorphism
$$\Gm\pi(G_i)\simeq  MT(W'),$$
where $\Gm$ is identified to the group of homotheties in $GL(W')\supset MT(W')$.
\end{enumerate}
\end{proposition}

\begin{proof}
Via the central isogeny $p$, we may consider $W'$ as an simple representation of $T\times G_1\times\ldots\times G_r.$ The irreducibility of $W'$ guarantees that we may write $W'$ as a tensor product
\begin{equation}\label{equation:decomposition-tensor}
W\simeq W_0\otimes\ldots\otimes W_r,
\end{equation}
where $W_0$ is an simple representation of $T$, and $W_i$ is an simple representation of $G_i$ for all $i$ between $1$ and $r$.

Via the fractional morphisms $\widetilde h_i$, the spaces $W_i$ are endowed with fractional Hodge structures, and the isomorphism \eqref{equation:decomposition-tensor} is an isomorphism of fractional Hodge structures.

Applying Proposition \ref{proposition:tensor-abelian} and the irreducibility of the $W_i$, we see that there exists a unique integer $i$ between $0$ and $r$ such that for all $j\neq i$, $W_j$ is isomorphic to the fractional Hodge structure $\Q(a_j)$ for some rational number $a_j$, and $W_i$ is a fractional Hodge structure of type $\{(a_i+1, a_i), (a_i, a_i+1)\}$ for some rational number $a_i$. 

Let $j\neq i$ be an integer between $1$ and $r$. Since $W_j$ is isomorphic to $\Q(a_j)$, $G_j$ acts on $W_j$ through a character. By assumption, this characteer is trivial, so that $a_j=0$ and $G_j$ acts trivially on $W_j$. This proves $(i)$ and $(ii)$.

With the notation of $(iii)$, since $G_i$ is simple, the restriction of the surjection $\pi$ to $G_i$ is a central isogeny. The surjection
$$\pi : G=MT(W)\lra MT(W')$$
maps all the $G_j$ to $1$ for $j\neq k$, and maps $T$ to $\mathbb G_m$ by \ref{proposition:tensor-abelian}. This proves $(iii)$.
\end{proof}

For any integer $i$ between $1$ and $r$, let $W_i$ denote the sum of those simple direct factors of $W$ on which $G_i$ acts nontrivially. For any $j\neq i$, the group $G_j$ acts trivially on $W_i$.

%
%

\begin{proposition}\label{proposition:faithful}
Let $i_1, \ldots, i_k$ be $k$ distinct integers between $1$ and $r$, and let $W_{i_1, \ldots, i_k}$ be the direct factor
$$W_{i_1, \ldots, i_k}=W_{i_1}\oplus\ldots\oplus W_{i_k}$$
of the Hodge structure $W$. Let
$$\pi : G=MT(W)\lra MT(W_{i_1, \ldots, i_k})$$
be the natural surjection. Then the restriction of $\pi$ to the subgroup $G_{i_1}\ldots G_{i_k}$ of $G$ is injective.
\end{proposition}

\begin{proof}
Since $W$ is semisimple, we may write 
$$W=W_{i_1, \ldots, i_k}\oplus W',$$
as a direct sum of Hodge structures. By assumption, the group $G_{i_1}\ldots G_{i_k}$ acts trivially on $W'$. Since the representation of $G_{i_1}\ldots G_{i_k}$ on $W$ is faithful by construction, this proves that $G_{i_1}\ldots G_{i_k}$ acts faithfully on $W_{i_1, \ldots, i_k}$, which proves the result.
\end{proof}

\subsection{Beginning of the proof}\label{subsection:beginning}

We keep the notation of Theorem \ref{theorem:torus-abelian} and start its proof.

\subsubsection{}\label{subsubsection:setting} After replacing $V$ with $V\otimes\Q(-c)$, we may assume that $c=0$, i.e., that $V$ is isomorphic to a subquotient of the Hodge structure 
$$W^{\otimes a}\otimes (W^\vee)^{\otimes b}.$$

Denote by $G$ (resp. $H$) the Mumford-Tate group of $W$ (resp. $V$), and by

$$h_G : \Sd\lra G_\R$$ 
and
$$h_H : \Sd\lra H_\R$$ 
the corresponding morphisms from the Deligne torus $\Sd$.

The group $G$ acts naturally on the space $W^{\otimes a}\otimes (W^{\vee})^{\otimes b}$. Together with $h_G$, this action defines the Hodge structure on $W^{\otimes a}\otimes (W^{\vee})^{\otimes b}$. Since $V$ is a subquotient of $W^{\otimes a}\otimes (W^{\vee})^{\otimes b}$, the group $G$ acts on $V$ in such a way that the action of $\Sd$ on $V_\R$ induced by $h_G$ defines the Hodge structure on $V$. As a consequence, the action of $G$ on $V$ defines a morphism of Mumford-Tate group $p : G\ra H$, and the diagram
\[
\xymatrix{
\Sd\ar[r]^{h_g}\ar[dr]^{h_H} & G_\R\ar[d]^{p_\R}\\
& H_\R
}
\]
is commutative. 

By construction, the image of $h_H(\C^*)$ in $H(\R)$ is Zariski-dense in $G$. As a consequence, $p$ is surjective.

\subsubsection{}\label{subsubsection:reductive}
Since $V$ is polarizable, its Mumford-Tate group $H$ is reductive. Let $U$ be the unipotent radical of $G$. Then $p(U)$ is a connected, unipotent soubgroup of $H$. It is normal since $p$ is surjective. This proves that $U$ is contained in the kernel of $p$ and that $p$ factors through the reductive quotient $G^{\rm{red}}$ of $G$. 

Consider the faithful representation $W^{\mathrm{ss}}$ of $G^{\rm{red}}$. It is semisimple of type $\{(0, 1), (1, 0)\}.$ Via the surjection $G^{\rm{red}}\ra H$, consider $V$ as a representation of $G^{\rm{red}}.$ By \cite[Chapter I, Proposition 3.1(a)]{DeligneMilneOgusShih}, the representation $V$ of $G^{\rm{red}}$ is isomorphic to a direct factor of $(W^{\mathrm{ss}})^{\otimes a'}\otimes (W^{ss, \vee})^{\otimes b'}$ for some nonnegative integers $a'$ and $b'$. In particular, the Hodge structure $V$ is isomorphic to a direct factor of $(W^{\mathrm{ss}})^{\otimes a'}\otimes (W^{\mathrm{ss}, \vee})^{\otimes b'}$.  

By the previous two paragraphs, we may replace $G$ by its reductive quotient $G^{\rm{red}}$ and $W$ by its semisimplification $W^{\mathrm{ss}}$. From now on, we assume that $G$ is reductive and $W$ is semisimple.

\subsubsection{} We apply the results of \ref{subsection:MT-ab}. Let $Z$ be the identity component of the center of $G$ and let $G^{\der}$ be the derived subgroup of $G$. The multiplication map induces a central isogeny
$$Z\times G^{\der}\lra G.$$
Additionnally, the group $G^{\der}$ is semisimple.

Let $Z'$ be the identity component of the center of $H$. The surjection $G\ra H$ induces a surjection $Z\ra Z'.$ As a consequence, we may find connected subgroups $T$ and $G_1$ of $Z$ such that the multiplication map 
$$T\times G_1\lra Z$$
is an isogeny, and the composition 
$$G_1\lra Z\lra Z'/w_{h_H}(\Gm)$$
is an isogeny.

By \cite[Proposition 1.1.14]{Deligne79}, the group 
$$(Z'/w_{h_H}(\Gm))(\R)$$
is compact. As a consequence, the group $G_1(\R)$ is compact.

Let $(G_i)_{2\leq i\leq r}$ denote the minimal connected normal subgroups of $G^{\mathrm{der}}$. Then the $G_i$ are simple normal subgroups of $G$ that commute pairwise, and the multiplication map induces a central isogeny
$$q : T\times G_1\ldots\times G_r\lra G.$$

After possibly reordering the groups $G_i$, we may find an integer $k\geq 1$ such that, for any $i$ between $1$ and $k$, the morphism 
$$p_{|G_i} : G_i\lra G^{\mathrm{der}}$$
is an isogeny onto its image, and, for any $i$ between $k+1$ and $r$, $p$ maps $G_i$ to the identity element of $G$. In particular, the restriction of $p\circ q$ to $G_1\times\ldots\times G_k$ is surjective onto $H/w_{h_H}(\Gm).$


\begin{lemma}\label{lemma:Cartan-involution}
For any integer $j$ between $2$ and $k$, conjugation by $h_G(i)$ on $G_\R$ induces a Cartan involution of $G_{j, \R}$, namely, the group
$$\{g\in G_j(\C),\, g=h(i)\overline g h(i)^{-1}\}$$
is compact.
\end{lemma}

\begin{proof}
Let $C$ (resp. $C'$) denote the involution of $G_\R$ (resp. $H_\R$) given by conjugation by $h_G(i)$ (resp. $h_H(i)$). Since $G_j$ is a normal subgroup of $G$, $G_{j, \R}$ is stable under $C$. 

By construction, the restriction of $p\circ q$ to $G_j$ defines an isogeny of $G_j$ onto a closed subgroup of $H^{\mathrm{ad}}$. In particular, it defines an isogeny from the group
$$G_j^{(C)}(\R) :=\{g\in G_j(\C),\, g=C(\overline g)\}$$
onto a closed subgroup of 
$$H^{\mathrm{ad}, (C')}(\R) :=\{h\in H^{\mathrm{ad}}(\C),\, h=C'(\overline h)\}.$$

Since $V$ is polarizable, \cite[Proposition 1.1.14]{Deligne79} shows that $C'$ is a Cartan involution of $H^{\mathrm{ad}}_\R$. This proves that $G_j^{(C)}(\R)$ is compact.
\end{proof}

\begin{proposition}\label{proposition:irreducible-polarizable}
Let $j$ be an integer between $1$ and $k$, and let $W'$ be an simple direct factor of the Hodge structure $W$ such that $G_j$ acts nontrivially on $W'$. Then $W'$ is polarizable.
\end{proposition}

\begin{proof}
Let 
$$\pi : G=MT(W)\lra MT(W')$$
denote the natural surjection. Let 
$$h' : \Sd\lra MT(W')_\R$$
be the morphism defining the Hodge structure on $W'$, so that $h'=\pi\circ h.$ 

Proposition \ref{proposition:unique-action} shows that $MT(W')$ is equal to $\Gm\pi(G_j)$, where $\Gm$ is identified to the group of homotheties of $W'$. Together with Lemma \ref{lemma:Cartan-involution} in case $j\geq 2$ and by construction if $j=1$, this implies that $h'(i)$ is a Cartan involution of $MT(W')/w_{h'}(\Gm)$, so that $W'$ is polarizable, see \cite[Proposition 2.11]{Deligne72} or \cite[Proposition 3.2]{Milne20}.
\end{proof}

For any integer $i$ between $1$ and $r$, let $W_j$ denote the sum of those simple direct factors of the Hodge structure $W$ on which $G_j$ acts nontrivially. Let $W_p$ be the Hodge structure 
$$W_p=\bigoplus_{j=1}^k W_j.$$

\begin{proposition}\label{proposition:polarization-exists}
The Hodge structure $W_p$ is polarizable.
\end{proposition}

\begin{proof}
This is a formal consequence of Proposition \ref{proposition:irreducible-polarizable}.
\end{proof}

Note that the surjections
$$G=MT(W)\lra MT(W_p)$$
and
$$G\lra H=MT(V)$$
endow both $V$ and $W_p$ with actions of $G$.
%
%

\subsection{Proof of Theorem \ref{theorem:torus-abelian}}

\begin{proposition}\label{proposition:injection-equivariant}
There exists nonnegative integers $a'$ and $b'$, and a linear injection 
$$\phi : V\hlra W_p^{\otimes a'}\otimes (W_p^\vee)^{\otimes b'}$$
that is equivariant with respect to the action of the subgroup $G_1\ldots G_k$ of $G$.
\end{proposition}

\begin{proof}
Proposition \ref{proposition:faithful} shows that $W_p$ is a faithful representation of the reductive group $G_1\ldots G_k$. The result follows from \cite[Chapter I, Proposition 3.1(a)]{DeligneMilneOgusShih}.
\end{proof}

\begin{proof}[Proof of Theorem \ref{theorem:torus-abelian}]
To prove Theorem \ref{theorem:torus-abelian}, it suffices to show that there exists an integer $c'$ such that the injection $\phi$ of Proposition \ref{proposition:injection-equivariant} induces an injection of Hodge structures
$$\psi : V\hlra W_p^{\otimes a'}\otimes (W_p^\vee)^{\otimes b'}\otimes\Q(c').$$

As both $V$ and $W_p^{\otimes a'}\otimes (W_p^\vee)^{\otimes b'}$ are pure Hodge structures, this is in turn equivalent to proving that for any simple sub-Hodge structure $V'$ of $V$, there exists an integer $c'$ such that $\phi$ induces an injection of Hodge structures 
$$V'\lra W_p^{\otimes a'}\otimes (W_p^\vee)^{\otimes b'}\otimes\Q(c').$$
Indeed, if this holds, then $c'$ is independent of $V'$ by purity. We pick such a $V'$. 

Let $H'$ be the image of the subgroup $G_1\ldots G_r$ in $H$. Then the multiplication map induces an isogeny
$$w_h(\Gm)\times H'\lra H,$$
and we may lift $h_H : \Sd\ra H_\R$ to a fractional morphism
$$\widetilde h' : \Sd\lra w_h(\Gm)_\R\times H'_\R.$$
We denote by $\widetilde h_{H'} : \Sd\ra H'_\R$ the component of $\widetilde h'$ mapping to $H'_\R$.

Since $V'$ is simple, the group $w_{h_H}(\Gm)$ acts on $V'$ by a homothety. This proves that the Hodge structure on $V'$ coincides up to a twist by some $\Q(c')$ with the fractional Hodge structure on $V'$ induced by $\widetilde h_{H'}.$

Similarly, consider the lift of $h_G$ to a fractional morphism 
$$\widetilde h : \Sd\lra (Z\times G_1\times\ldots\times G_r)_\R,$$
as well as its component $\widetilde h_{1,\ldots, k} : \Sd\lra (G_1\times\ldots\times G_k)_\R.$ 

By construction, $Z$ acts through a character on $W_p$, and $G_{k+1}, \ldots, G_r$ all act trivially on $W_p$. Consequently, the same argument as above shows that if $W'_p$ is an simple sub-Hodge structure of $W_p$, the Hodge structure on $W'_p$ coincides up to a twist by some $\Q(d')$ with the fractional Hodge structure on $W'_p$ induced by $\widetilde h_{1,\ldots, k}.$

Finally, the previous paragraph implies that if $W''$ is an simple sub-Hodge structure of $W_p^{\otimes a'}\otimes (W_p^\vee)^{\otimes b'}$,  the Hodge structure on $W''$ coincides up to a twist by some $\Q(e')$ with the fractional Hodge structure on $W''$ induced by $\widetilde h_{1,\ldots, r}.$

The equivariance of $\phi$ with respect to the action of $G_1\ldots G_k$ proves that $\phi$ defines a morphism between the fractional Hodge structures on $V$ and $W_p^{\otimes a'}\otimes (W_p^\vee)^{\otimes b'}$ defined by $\widetilde h_{H'}$ and $\widetilde h_{1,\ldots, k}$ respectively. This finishes the proof of Theorem \ref{theorem:torus-abelian}.
\end{proof}

\section{Universality of the Kuga-Satake construction}

\subsection{The Kuga-Satake construction}

\subsubsection{}

We recall briefly the Kuga-Satake construction in the context of Hodge structures that are not necessarily polarizable. We refer to \cite{Deligne72} and \cite[Chapter 4]{HuybrechtsK3book} for more details, see in particular \cite[Chapter 4, Remark 2.3]{HuybrechtsK3book} for the non polarized case.

\begin{definition}
A \emph{Hodge structure of $K3$ type} is a pure Hodge structure $V$ of weight $0$, of type $\{(1, -1), (0,0), (-1, 1)\}$ with $\dim V^{1, -1}=1$.

A \emph{Beauville-Bogomolov} form on $V$ is a nondegenerate quadratic form $q$ on $V$ that induces an morphism of Hodge structures
$$q : V\otimes V\lra \Q$$
and is definite positive on the real part of $V^{1,-1}\oplus V^{-1, 1}.$
\end{definition}

Note that a Beauville-Bogomolov form does not define a polarization on $V$ in general. If $X$ is a compact hyperkähler manifold, the Hodge structure $H^2(X, \Q)$ is of $K3$ type, endowed with a natural Beauville-Bogomolov form, see \cite{Beauville83}.

Let $V$ be a Hodge structure of $K3$ type endowed with a Beauville-Bogomolov form. The natural morphism 
$$h : \Sd\lra GL(V)_\R$$
defining the Hodge structure on $V$ factors through the special orthogonal group $SO(V)_\R$.

Let $C(V)$ denote the Clifford algebra of $V$, $C^+(V)$ its even part. We regard $V$ as a subspace of $C(V)$. Let $\mathrm{CSpin}(V)$ denote the Clifford group of $V$, defined as the algebraic groups of invertible elements $g$ of $C^+(V)$ such that 
$$gVg^{-1}=V.$$
Conjugation by elements of $\mathrm{CSpin}(V)$ defines a surjection 
$$\mathrm{CSpin}(V)\lra SO(V)$$
with kernel reduced to the scalars. The computation of \cite[Chapter 4, 2.1]{HuybrechtsK3book} shows that the morphism $h$ lifts uniquely to a morphism
$$h' : \Sd\lra \mathrm{CSpin}(V)_\R.$$

Consider the representation of the group $\mathrm{CSpin}(V)$ on $C^+(V)$ by multiplication on the left. By \cite[Proposition 4.5]{Deligne72}, the Hodge structure induced by $h'$ and this representation on $C^{+}(V)$ has type $\{(0,1), (1, 0)\}.$ This is the \emph{Kuga-Satake Hodge structure} associated to $(V, q)$. It is known to be polarizable if $q$ defines a polarization on $V$. We denote it by $H_{KS}$.

\subsubsection{}

Let $n$ be the dimension of $V$. The dimension of the algebra $C^+(V)$ is $2^{n-1}$. By e.g. \cite[3.4]{Deligne72}, we may understand the representation of $\mathrm{CSpin}(V_\C)$ on $C^+(V_\C)$ by multiplication on the left as follows.

If $n$ is odd, then $C^+(V_\C)$ is a direct sum of $2^{(n-1)/2}$ copies of the spin representation, which is irreducible of dimension $2^{(n-1)/2}$.

If $n$ is even, then $C^+(V_\C)$ is a direct sum of $2^{n/2-1}$ copies of the direct sum of the two half-spin representations, which are both irreducible of dimension $2^{n/2-1}$.

Assume that the Mumford-Tate group of the Hodge structure $V$ is the full special orthogonal group $SO(V)$. Considering the commutative diagram 
\[
\xymatrix{
\Sd\ar[r]^-{h'}\ar[dr]^h & \mathrm{CSpin}(V)_\R\ar[d]\\
& SO(V)_\R
}
\]
in which the vertical map is surjective shows that the Mumford-Tate group of $H_{KS}$ is the group $\mathrm{CSpin}(V)$. In particular, it is reductive, so that the Hodge structure $H_{KS}$ is semisimple. We want to give a description of the simple factors of $H_{KS}$. The simple factors of $H_{KS}$ are exactly the simple factors of the representation of $\mathrm{Cspin}(V)$ on $H_{KS}$. 

\begin{lemma}\label{lemma:simple-basic}
Let $H$ and $H'$ be two nonisomorphic simple factors of the Hodge structure $H_{KS}$. 

Then $n$ is even, there exists a positive integer $r$ such that the representation of $\mathrm{CSpin}(V_\C)$ on $H_\C$ is a direct sum of copies of one of the half-spin representations, and the representation of $\mathrm{CSpin}(V_\C)$ on $H'_\C$ is a direct sum of copies of the other half-spin representation. \end{lemma}

\begin{proof}
Since $H$ and $H'$ are nonisomorphic and simple, there are no nonzero $\mathrm{CSpin}(V)$-equivariant morphisms from $H$ to $H'$. As a consequence, there are no nonzero $\mathrm{CSpin}(V_\C)$-equivariant morphisms from $H_\C$ to $H'_\C$. This immediately implies the result by the above description of the representation $H_{KS, \C}$ of $\mathrm{CSpin(V_\C)}$.
\end{proof}

The following theorem is an elaboration on \cite[Section 8]{vanGeemen98}.

\begin{theorem}\label{theorem:simple-factors}
Let $\delta$ be the discriminant of the quadratic form $q$. Then the following holds.
\begin{enumerate}[(i)]
\item If $n$ is odd, then there exists a simple Hodge structure $H$ and a positive integer such that 
$$H_{KS}\simeq H^{\oplus N}.$$
There exists $r\in\{1, 2\}$ such that the representation of $\mathrm{CSpin}(V_\C)$ on $H_\C$ is a direct sum of $r$ copies of the spin representation. If $r=1$, then $H$ has dimension $2^{(n-1)/2}$ and $N=2^{(n-1)/2}.$ If $r=2$, then $H$ has dimension $2^{(n+1)/2}$ and $N=2^{(n-3)/2}.$

\item If $n$ is even and $(-1)^{n/2}\delta$ is not a square in $\Q$, then there exists a simple Hodge structure $H$ and a positive integer such that 
$$H_{KS}\simeq H^{\oplus N}.$$
There exists $r\in\{1, 2\}$ such that the representation of $\mathrm{CSpin}(V_\C)$ on $H_\C$ is a direct sum of $r$ copies of the direct sum of the two half-spin representations. If $r=1$, then $H$ has dimension $2^{n/2}$ and $N=2^{n/2-1}.$ If $r=2$, then $H$ has dimension $2^{n/2+1}$ and $N=2^{n/2-2}.$

\item If $n$ is even and $(-1)^{n/2}\delta$ is a square in $\Q$, then there exist two nonisomorphic simple Hodge structures $H$ and $H'$, and a positive integer $N$ such that
$$H_{KS}\simeq (H\oplus H')^{\oplus N}.$$
There exists a positive integer $r\in\{1, 2\}$ such that the representation of $\mathrm{CSpin}(V_\C)$ on $H_\C$ is a direct sum of $r$ copies of one of the half-spin representations, and the representation of $\mathrm{CSpin}(V_\C)$ on $H'_\C$ is a direct sum of $r$ copies of the other half-spin representation. If $r=1$, then $H$ and $H'$ have dimension $2^{n/2-1}$ and $N=2^{n/2-1}$. If $r=2$, then $H$ and $H'$ have dimension $2^{n/2}$ and $N=2^{n/2-2}$. 
\end{enumerate}
\end{theorem}

In \cite{vanGeemen98}, it is shown that all the situations above do occur.

\begin{proof}
Lemma \ref{lemma:simple-basic} shows that there are at most two nonisomorphic simple factors of the Hodge structure $H_{KS}$, and that if $n$ is odd, there is only one simple factor. 

We proceed to investigate the structure of the endomorphism algebra of the Hodge structure $H_{KS}$. Since the Mumford-Tate group of $H_{KS}$ is $\mathrm{CSpin}(V)$, we have:
$$\mathrm{End}_{\mathrm{Hodge}}(H_{KS})=\mathrm{End}_{\mathrm{CSpin}(V)}(C^+(V)),$$
where $\mathrm{CSpin}(V)$ acts on the Clifford algebra $C^+(V)$ by multiplication on the left. This latter algebra is computed in \cite[Theorem 7.7]{vanGeemen98}.

\bigskip

Assume that $n$ is odd. Then
$$\mathrm{End}_{\mathrm{Hodge}}(H_{KS})\simeq M_{2^{(n-3)/2}}(D),$$
where $D$ is a quaternion algebra over $\Q$. This proves that, in case $(i)$, we have 
$$H_{KS}\simeq H^{\oplus 2^{(n-3)/2}}$$
or
$$H_{KS}\simeq H^{\oplus 2^{(n-1)/2}}$$
depending on whether or not $D$ is split. Accordingly, the dimension of $H$ is either $2^{(n+1)/2}$ or $2^{(n-1)/2}$, so that $H$ is the sum of $1$ or $2$ copies of the spin representation. This is the case described in $(i)$.

\bigskip

Assume that $n$ is even and $(-1)^{n/2}\delta$ is not a square in $\Q$.  Then 
$$\mathrm{End}_{\mathrm{Hodge}}(H_{KS})\simeq M_{2^{n/2-1}}(D),$$
where $D$ is a quaternion algebra over $\Q(\sqrt{(-1)^{n/2}\delta}).$

If $D$ is split, then we argue as above to show that 
$$H_{KS}\simeq H^{\oplus 2^{n/2}},$$
where $H$ is simple and $H_\C$ is isomorphic to the sum of two copies of the sum of the two half-spin representations. 

If $D$ is nonsplit, then 
$$H_{KS}\simeq H^{\oplus 2^{n/2}-2},$$
where $H$ is simple and $H_\C$ is isomorphic to the sum of the two half-spin representations. This is the case described in $(ii)$.

\bigskip

Assume that $n$ is even and $(-1)^{n/2}\delta$ is a square in $\Q$. Then 
$$\mathrm{End}_{\mathrm{Hodge}}(H_{KS})\simeq M_{2^{n/2-2}}(D)\times M_{2^{n/2-2}}(D),$$
where $D$ is a quaternion algebra over $\Q.$ 

If $D$ is split, then 
$$\mathrm{End}_{\mathrm{Hodge}}(H_{KS})\simeq M_{2^{n/2-1}}(\Q)\times M_{2^{n/2-1}}(\Q),$$
so that 
$$H_{KS}\simeq (H\oplus H')^{\oplus 2^{n/2-1}},$$
where $H$ and $H'$ are two nonisomorphic simple factors of $H_{KS}$. Lemma \ref{lemma:simple-basic} shows that $H_\C$ is a sum of copies of one of the half-spin representation, and $H'_\C$ a sum of copies of the other. The structure of the representation $H_{KS, \C}$ implies that $H_\C$ is actually isomorphic to one of the half-spin representations, and $H'_\C$ to the other. 

If $D$ is nonsplit, the same argument shows that
$$H_{KS}\simeq (H\oplus H')^{\oplus 2^{n/2-2}},$$
where $H$ and $H'$ are two nonisomorphic simple factors of $H_{KS}$, $H_\C$ is isomorphic to the sum of two copies of one of the half-spin representation, and $H'_\C$ to the sum of two copies of the other. This is the case described in $(iii)$.
\end{proof}

\begin{example}
Let $T$ be a very general complex torus of dimension $4$, and let 
$$V=H^2(T, \Q)$$
endowed with the quadratic form $q$ given by cup-product. Then $(V, q)$ is isomorphic to the sum of three copies of the hyperbolic plane, so that the discriminant $\delta$ of $q$ is $-1$. Since $V$ is $6$-dimensional, this proves that we are in the situation of $(iii)$ above. It is possible to show that the simple factors of $H_{KS}$ are $H^1(T, \Q)$ and its dual, which are nonisomorphic -- note that $T$ is not an abelian variety. The reader may consult \cite{Morrison85} for the related case of abelian surfaces.
\end{example}

\begin{example}
Let $X$ be a holomorphic symplectic variety. Then $H^2(X, \Q)$ endowed with its Beauville-Bogomolov quadratic form $q$ is a Hodge structure of type $K3$. The quadratic form $q$ has signature $(3, b_2(X)-3)$, where $b_2(X)$ is the dimension of $H^2(X, \Q)$. Let $\delta$ be the discriminant of $q$. 

Assume that $b_2(X)$ is odd, and $V=H^2(X, \Q)$. Then we are in case $(i)$  and the unique simple factor of $H_{KS}$ has dimension $2^{(b_2(X)+1)/2}$ or $2^{(b_2(X)-1)/2}$.

Assume that $b_2(X)$ is divisible by $4$, and $V=H^2(X, \Q)$. Then $\delta$ is negative, and $(-1)^{b_2(X)/2}\delta$ is negative, so that we are in case $(ii)$  and the unique simple factor of $H_{KS}$ has dimension $2^{b_2(X)/2}$ or $2^{b_2(X)/2+1}$.

Assume that $b_2(X)$ is even, $X$ is equipped with an ample line bundle, and $V$ is the orthogonal of the polarization in $H^2(X, \Q)$. Then we are in case $(i)$  and the unique simple factor of $H_{KS}$ has dimension $2^{b_2(X)/2-1}$ or $2^{b_2(X)/2}$.

Assume that $b_2(X)$ is congruent to $3$ modulo $4$, $X$ is equipped with an ample line bundle, and $V$ is the orthogonal of the polarization in $H^2(X, \Q)$. Then the discriminant of the restriction of $q$ to $V$ is positive, so that we are in case $(ii)$ again, and the unique simple factor of $H_{KS}$ has dimension $2^{(b_2(X)-1)/2}$ or $2^{(b_2(X)+1)/2}$.

\end{example}

Two special cases in even dimension are as follows. 
%
%
%
%
%
%

\subsection{A preliminary group-theoretic result}

Let $V_\C$ be a finite-dimensional  complex vector space endowed with a nondegenerate quadratic form $q_\C$. Let $u$ and $v$ be two elements of $V$ with 
$$q(u)=q(v)=1$$
and
$$q(u, v)=0.$$
Let 
$$\nu : \mathbb G_{m, \C}\lra SO(V_\C),\, z\mapsto \nu(z)$$
be the cocharacter that acts on $u$ (resp. $v$) by multiplication by $z$ (resp. $z^{-1}$) and acts trivially on the orthogonal of the $2$-plane generated by $u$ and $v$. 

It is readily checked that the action of $\mathbb G_{m, \C}$ on $\mathrm{Lie}\,SO(V_\C)$ through $\nu$ and the adjoint action factors through the characters $z$, $z^{-1}$ and $1$. As a consequence of \cite[1.2.5]{Deligne79}, see also \cite[2.3]{Milne13}, we may attach to the conjugacy class of $\nu$ a \emph{special vertex} of the Dynkin diagram of $SO(V_\C)$ as follows. Choose a maximal torus $T$ of $SO(V_\C)$ through which $\nu$ factors, $S$ a basis for the root system of $T$, $R^+$ the corresponding set of positive roots. After conjugating $\nu$, we may assume that the integers $\langle \alpha, \nu\rangle$ are nonnegative for $\alpha\in R^+$. Then there exists a unique simple root $\alpha$ for which 
$$\langle \alpha, \nu\rangle=1.$$
The corresponding vertex of the Dynkin diagram of $SO(V_\C)$ is the special vertex associated to $\nu.$

\begin{lemma}\label{lemma:special-vertex}
The special vertex associated to $\nu$ is the leftmost vertex of the Dynkin diagram of $SO(V_\C)$.
\end{lemma}

\begin{proof}
We assume that the dimension of $V_\C$ is an even number $2m$, the odd case being similar. Consider a basis $e_1,\ldots, e_m, f_1, \ldots, f_m$ of $V_\C$ with $e_1=u, f_1=v$, and, for all $i, j$, 
$$q(e_i, e_j)=q(f_i, f_j)=0$$
and
$$q(e_i, f_j)=\delta^i_j.$$
Let $T$ be the maximal torus of $SO(V_\C)$ consisting of diagonal matrices with diagonal coefficients 
$$(t_1, \ldots, t_m, t_1^{-1}, \ldots, t_m^{-1})$$
and identify the $t_i$ with characters of $T$. Then we may take for $R^+$ the simple roots
$$\alpha_1=t_1/t_2,\, \ldots,\, \alpha_{n-1}=t_{n-1}/t_n,\, \alpha_n=t_{n-1}t_n.$$
We obtain
$$\langle \alpha_1, \nu\rangle=1$$
and
$$\langle \alpha_i, \nu\rangle=0$$
for $i>1.$ This proves that the special vertex corresponds to the simple root $\alpha_1$.
\end{proof}

We keep the notation above.

\begin{proposition}\label{proposition:classification}
Let 
$$G_{1, \C}\lra SO(V_\C)$$
be an isogeny of complex algebraic groups and let
$$\widetilde \nu : \widetilde{\mathbb G_{m, \C}}\lra G_{1, \C}$$
be the fractional cocharacter of $G_{1, \C}$ lifting $\nu$. Let $W$ be an simple representation of $G_{1, \C}$ such that the action of $\widetilde{\mathbb G_{m, \C}}$ on $W$ via $\widetilde\nu$ has only two weights $a$ and $a+1$ for some rational number $a$. Then $G_{1, \C}$ is the universal cover of $SO(V_\C)$. Furthermore, if $\dim V_\C$ is odd, then $W$ is the spin representation of $G_1$, and, if $\dim V_\C$ is even, $W$ is one of the two half-spin representations of $G_1.$
\end{proposition}

\begin{proof}
The classification of all complex representations of $G_1$ as in the statement of the proposition is due to Satake \cite{Satake65}, and is explained in \cite[1.3.5--1.3.9]{Deligne79}. In the end, table 1.3.9 there, or the last table in \cite{Satake65} proves the result -- indeed, Lemma \ref{lemma:special-vertex} proves that, we the notation of Deligne, we are only considering the diagrams $B_n$ and $D_n^\R$. The reader may also consult \cite[Section 2 and Section 10]{Milne13} for a more detailed discussion, where the Dynkin diagram equipped with a special vertex is denoted by $D_n(1)$.
\end{proof}

\subsection{}

The goal of this section is to prove the following result. 

\begin{theorem}\label{theorem:KS-univ}
Let $V$ be a Hodge structure of $K3$ type, of dimension $n$, let $q$ be a Beauville-Bogomolov form for $V$ and let $T$ be a complex torus. Assume that there exist integers $a$, $b$ and $c$ together with an injective morphism of Hodge structures
$$V\hlra H^1(T, \Q)^{\otimes a}\otimes (H^1(T, \Q)^\vee)^{\otimes b}\otimes\Q(c)$$
with $a$ and $b$ nonnegative.

Assume that the Mumford-Tate group of $V$ is the full special orthogonal group $SO(V)$. Then $T$ contains a simple factor of the Kuga-Satake variety of $X$ as a subquotient up to isogeny. 

In particular, the dimension of $T$ is at least $2^{n/2-2}$ if $n$ is even, and $2^{(n-3)/2}$ if $n$ is odd. If $n$ is even and $(-1)^{n/2}\delta$ is not a square in $\Q$, where $\delta$ is the discriminant of $q$, then the dimension of $T$ is at least $2^{n/2-1}.$
\end{theorem}

\begin{proof}
Let $H=H^1(T, \Q)$, and let $G$ be the Mumford-Tate group of $H$. By assumption, the Mumford-Tate group of $V$ is reductive. Arguing as in \ref{subsubsection:reductive}, we may replace $H$ with its semisimplification and assume that $G$ is reductive.

Let $\mathfrak g$ be the Lie algebra of $G$. As in \ref{subsubsection:setting}, we obtain a surjection 
$$p : G\lra SO(V)=MT(V)$$
and a commutative diagram
\[
\xymatrix{
\Sd\ar[r]^{h_G}\ar[dr]^h & G_\R\ar[d]^{p_\R}\\
& SO(V)_\R
}
\]
where $h$ and $h_G$ are the morphisms defining the Hodge structures on $V$ and $H$ respectively. Let 
$$\nu : \mathbb G_{m, \C}\lra SO(V)_\C$$
be the cocharacter defined as the composition
$$\mathbb G_{m, \C}\stackrel{\mu}{\lra} \Sd_\C\stackrel{h_\C}{\lra} SO(V)_\C.$$

We may find connected normal subgroups $G_1$ and $G_2$ that commute such that the multiplication map $G_1\times G_2\ra G$ is an isogeny and the restriction of $p$ to $G_1$ is an isogeny onto $SO(V)$. Let 
$$\widetilde h : \Sd\lra (G_1\times G_2)_\R$$
be the fractional lift of $h_G$, and let $\widetilde h_1$ be the component of $\widetilde h$ mapping to $G_{1, \R}$. 

By Proposition \ref{proposition:unique-action}, we may find a nonzero sub-Hodge structure $H_1$ of $H$ such that the fractional Hodge structure induced by $\widetilde h_1$ on $H_1$ has type $\{(a+1, a), (a, a+1)\}$ for some rational number $a$. In particular, the fractional cocharacter
$$\widetilde \nu : \mathbb G_{m, \C}\lra G_{1, \C}$$
obtained by lifting $\nu$ has only two weights $a$ and $a+1$ on the representation $H_{1, \C}$ of $G_{1, \C}$.

If the dimension $n$ of $V$ is odd (resp. even), Proposition \ref{proposition:classification} shows that $H_1$ contains the spin representation (resp. one of the half-spin representations). In particular, there is a nonzero $\mathrm{CSpin(V_\C)}$-equivariant morphism from $H_{1, \C}$ to $H_{KS, \C}$, so that is a nonzero $\mathrm{CSpin(V)}$-equivariant morphism from $H_{1}$ to $H_{KS}.$ In particular, these two Hodge structures have a nonzero simple factor in common. 

The statement on the dimension on $T$ is now a direct consequence of Theorem \ref{theorem:simple-factors}.
\end{proof}

\bibliographystyle{alpha}
\bibliography{KS_Debarre}

\end{document}